\documentclass{amsart}

\usepackage{amsmath}
\usepackage{amssymb}
\usepackage{amsthm}

\newcommand{\theoremref}[1]{\hyperref[#1]{Theorem~\ref*{#1}}}
\newcommand{\claimref}[1]{\hyperref[#1]{Claim~\ref*{#1}}}
\newcommand{\situationref}[1]{\hyperref[#1]{Situation~\ref*{#1}}}
\newcommand{\lemmaref}[1]{\hyperref[#1]{Lemma~\ref*{#1}}}
\newcommand{\definitionref}[1]{\hyperref[#1]{Definition~\ref*{#1}}}
\newcommand{\propositionref}[1]{\hyperref[#1]{Proposition~\ref*{#1}}}
\newcommand{\conjectureref}[1]{\hyperref[#1]{Conjecture~\ref*{#1}}}
\newcommand{\corollaryref}[1]{\hyperref[#1]{Corollary~\ref*{#1}}}
\newcommand{\exerciseref}[1]{\hyperref[#1]{Exercise~\ref*{#1}}}

\usepackage[mathscr]{euscript}

\usepackage[all]{xy}

\xyoption{2cell}
\UseAllTwocells

\usepackage{multicol}

\usepackage{xr-hyper}

\usepackage{hyperref}
\usepackage[usenames,dvipsnames]{xcolor}
\hypersetup{colorlinks=true,citecolor=NavyBlue,linkcolor=BrickRed,urlcolor=Orange}

\numberwithin{equation}{section}

\theoremstyle{plain}
\newtheorem{theorem}[equation]{Theorem}
\newtheorem{proposition}[equation]{Proposition}
\newtheorem{lemma}[equation]{Lemma}

\newtheorem{corollary}[equation]{Corollary}

\newtheorem{proposition/definition}[equation]{Proposition/Definition}

\theoremstyle{definition}

\theoremstyle{remark}
\newtheorem{remark}[equation]{Remark}

\newtheorem{notation}[equation]{Notation}

\usepackage{tikz}
\usepackage{tikz-cd}
\usetikzlibrary{matrix,arrows}
\usetikzlibrary{decorations.pathmorphing}

\DeclareMathOperator{\Pic}{Pic}

\DeclareMathOperator{\Sym}{Sym}

\usepackage{amscd}

\newcommand{\git}{\mathbin{
  \mathchoice{/\mkern-6mu/}
    {/\mkern-6mu/}
    {/\mkern-5mu/}
    {/\mkern-5mu/}}}

\def\Pic{\operatorname{Pic}}

\def\Hom{\operatorname{Hom}}

\def\Sym{\operatorname{Sym}}

\newcommand{\bP}{\mathbb{P}}

\newcommand{\bZ}{\mathbb{Z}}
\newcommand{\bQ}{\mathbb{Q}}
\newcommand{\bR}{\mathbb{R}}
\newcommand{\bC}{\mathbb{C}}

\newcommand{\Aut}{\mathrm{Aut}}

\newcommand{\Res}{\mathrm{Res}}

\newcommand{\id}{\mathrm{id}}

\newcommand{\SO}{\mathrm{SO}}

\newcommand{\SL}{\mathrm{SL}}

\newcommand{\PGL}{\mathrm{PGL}}

\newcommand{\calD}{\mathcal{D}} 

\newcommand{\calO}{\mathcal{O}}
\newcommand{\calL}{\mathcal{L}}

\newcommand{\calF}{\mathcal{F}}

\newcommand{\calT}{\mathcal{T}}  

\newcommand{\calM}{\mathcal{M}}
\newcommand{\calP}{\mathcal{P}}
\newcommand{\calU}{\mathcal{U}}

\begin{document}

\title[Special Horikawa surfaces]{A generic global Torelli theorem for certain Horikawa surfaces}
\author[G. Pearlstein]{Gregory Pearlstein}
\address{Department of Mathematics, Texas A\&M University, College Station, TX 77843-3368}
\email{gpearl@math.tamu.edu}
\author[Z. Zhang]{Zheng Zhang}
\address{Department of Mathematics, Texas A\&M University, College Station, TX 77843-3368}
\email{zzhang@math.tamu.edu}
\date{\today}
\thanks{The authors acknowledge partial support from NSF Grant DMS-1361120.}

\bibliographystyle{amsalpha}
\maketitle
 
\begin{abstract}
Algebraic surfaces of general type with $q=0$, $p_g=2$ and $K^2=1$ were described by Enriques \cite{Enriques_LeSA} and then studied in more detail by Horikawa \cite{Horikawa_II}. In this paper we consider a $16$-dimensional family of special Horikawa surfaces which are certain bidouble covers of $\bP^2$. The construction is motivated by that of special Kunev surfaces (cf. \cite{Kynev} \cite{Catanese_kunev} \cite{Catanese_torellifail} \cite{Todorov_kunev}) which are counterexamples for infinitesimal Torelli and generic global Torelli problem. The main result of the paper is a generic global Torelli theorem for special Horikawa surfaces. To prove the theorem, we relate the periods of special Horikawa surfaces to the periods of  certain lattice polarized $K3$ surfaces using eigenperiod maps (see \cite{DK_ball}) and then apply a Torelli type result proved by Laza \cite{Laza_n16}. 
\end{abstract}

\section*{Introduction}
There are two particular situations where the period map plays an essential role for studying moduli spaces, namely principally polarized abelian varieties and (lattice) polarized $K3$ surfaces. In these cases, the period domains are Hermitian symmetric domains and the period maps are both injective and dominant. It is an interesting problem to find more examples where the period maps are injective and the images lie in certain Mumford-Tate subdomains which are locally Hermitian symmetric (but the Griffiths infinitesimal period relations may be non-trivial on the ambient period domains). We mention the examples previously studied by Allcock, Carlson, and Toledo \cite{ACT2, ACT}, Kondo \cite{Kondo}, Borcea \cite{Borcea}, Voisin \cite{Voisin_cy}, Rohde \cite{Rohde}, Garbagnati and van Geemen (\cite{vGG_cyshimura}).

The general problem of determining whether the period map is injective is called the Torelli problem. There are four types of Torelli problem (we follow  the terminology of \cite{Catanese_torelli}): local Torelli (whether the differential of the period map is injective), infinitesimal Torelli (local Torelli for the semi-universal deformation), global Torelli (whether the period map is injective) and generic global Torelli (whether the period map is injective over a open dense subset). 

Various Torelli theorems have been proved for a large class of varieties (see for example \cite{Catanese_torelli}). However, Kunev \cite{Kynev} constructed a counterexample for the infinitesimal Torelli and generic global Torelli problem (see also \cite{Catanese_kunev} \cite{Catanese_torellifail} \cite{Todorov_kunev}). Let us briefly recall the construction. Let $C_1$, $C_2$ be two smooth plane cubic curves intersecting transversely and $L$ be a general line. Let $X$ be the $(\bZ/2\bZ)^2$-cover of $\bP^2$ branched along $C_1+C_2+L$. Then $X$ is a minimal algebraic surface with $q(X)=0$, $p_g(X)=1$ and $K_X^2=1$. Following \cite{Catanese_kunev} such surfaces $X$ are {\sl special Kunev surfaces} whose bicanonical maps are Galois covers of $\bP^2$. The infinitesimal period map and period map for special Kunev surfaces both have $2$-dimensional fibers (the rough reason is that the period map only sees the Hodge structures of the intermediate $K3$ surfaces obtained as the desingularizations of the double planes along $C_1+C_2$).

One may ask whether is it possible to modify the construction or the period map for Kunev surfaces to get Torelli. Usui \cite{Usui_vmhs} and Letizia \cite{Letizia} considered the complement of the canonical curve $\Lambda \subset X$ and the mixed Hodge structure $H^2(X-\Lambda)$, and proved infinitesimal mixed Torelli and generic global mixed Torelli for special Kunev surfaces $X$ respectively. Our idea is to modify the branch data (see also \cite[Def. 2.1]{Pardini_cover}) $C_1+C_2+L$. Specifically, we consider the $(\bZ/2\bZ)^2$-cover $S$ of $\bP^2$ along a smooth quintic $C$ together with two generic lines $L_0$ and $L_1$. The surfaces $S$ are minimal surfaces with $q(S)=0$, $p_g(S)=2$ and $K_S^2=1$ which have been studied by Horikawa \cite{Horikawa_II}. We shall call such bidouble covers $S$ {\sl special Horikawa surfaces}. We should remark that these surfaces $S$ are also mentioned in the recent preprint \cite{Garbagnati_doublecoverk3} by Garbagnati. But her perspective (classify the possible branch loci of a smooth double cover of a $K3$ surface) is quite different from ours. 

Let us explain why we want to modify the branch data in this way. On one hand, there are two (lattice polarized) $K3$ surfaces hidden in the construction of a special Horikawa surface $S$ (this is also observed in \cite{Garbagnati_doublecoverk3}). Namely, one resolves the singularities of the double cover of $\bP^2$ branched along $C+L_1$ (resp. $C+L_0$) and gets $X_0$ (resp. $X_1$) which is a $K3$ surface. On the other hand, it is also natural to study the action of the Galois group $(\bZ/2\bZ)^2$ on the periods of the bidouble covers $S$. We shall show that the eigenperiods (cf. \cite[\S7]{DK_ball}) of $S$ are determined by the Hodge structures of the $K3$ surfaces $X_0$ and $X_1$, and apply the global Torelli theorem for (lattice polarized) $K3$ surfaces (a modified version is needed, see the next paragraph) to prove a generic global Torelli theorem for special Horikawa surfaces.

By the work of Catanese \cite{Catanese_bidouble} and Pardini \cite{Pardini_cover} the isomorphism classes of the bidouble covers $S$ are determined by the branch data $C+L_0+L_1$. This can be used to construct the coarse moduli space $\calM$ for special Horikawa surfaces (the moduli of special Kunev surfaces has been constructed in a similar manner, see \cite{Usui_type1kunev}). It also follows that we need the global Torelli theorem \cite[Thm. 4.1]{Laza_n16} for degree $5$ pairs $(C,L)$ consisting of plane quintics $C$ and lines $L$ (up to projective equivalence). A key point is that one needs to choose a suitable arithmetic group as explained in \cite[Prop. 4.22]{Laza_n16}.  

A typical way to prove a generic global Torelli is to study the infinitesimal variation of Hodge structure and first prove a variational Torelli theorem (cf. \cite{CDT_variational}). Our approach is different. One advantage is that we are able to describe the open dense subset over which the period map is injective explicitly. We suspect that variational Torelli fails for special Horikawa surfaces (otherwise by op. cit. the global Torelli holds for any discrete group of the automorphism group of the period domain, as long as the period map is well-defined, which seems not true; see also \cite{Hayashi_variational}).

After ``labeling" the lines $L_0$ and $L_1$, we obtain a period map (using the period maps for the degree $5$ pairs $(C,L_0)$ and $(C, L_1)$) from (a double cover of) the moduli $\calM$ of special Horikawa surfaces $S$ to a product of two arithmetic quotients of Type IV domains. The period map is generically injective. Therefore, special Horikawa surfaces are along the lines of the examples mentioned at the beginning of the paper. In an ongoing project with Gallardo and Laza, we use this period map as our guide to compactify the moduli space of special Horikawa surfaces.

A few words on the structure of the paper. The construction of special Horikawa surfaces is given in Section \ref{sec_construction}. As is well-known (cf. \cite{Horikawa_II}) the canonical model of an algebraic surface with $q=0$, $p_g=2$ and $K^2=1$ is a degree $10$ hypersurface in the weighted projective space $\bP(1,1,2,5)$. We shall give the equations for (the canonical models of) special Horikawa surfaces and use Griffiths residue to study the decomposition of the Hodge structures. The infinitesimal Torelli problem will be discussed in Section \ref{sec_inftorelli}. Usui \cite{Usui_wci} has proved the infinitesimal Torelli theorem for nonsingular weighted complete intersections satisfying certain conditions which will be checked for special Horikawa surfaces. A second proof will also be included which can be viewed as a boundary case of \cite[Thm. 3.1]{Pardini_surface} or \cite[Thm. 4.2]{Pardini_torelli} and might be of independent interest. In Section \ref{sec_globaltorelli} we discuss the generic global Torelli problem for special Horikawa surfaces and prove the main result. 

\subsection*{Acknowledgement} The work is partly motivated by the recent project by Green, Griffiths, Laza and Robles on studying degenerations of ``H-surfaces" (which are of general type with $p_g=K^2=2$) using Hodge theory. We thank Griffiths for his interest in this paper. We also thank the referee for the valuable comments. Finally,  we are grateful to P. Gallardo and R. Laza for several useful discussions.

\section{Special Horikawa surfaces} \label{sec_construction}
Let $C$ be a smooth plane quintic curve. Let $L_0$, $L_1$ be two distinct lines which intersect $C$ transversely and satisfy that $C \cap L_0 \cap L_1 = \emptyset$. We are interested in the bidouble cover $S$ of $\bP^2$ branched along $C+L_0+L_1$. 

Specifically, the surface $S$ can be constructed in the following way.  Take the double cover $\bar{X}_0$ of $\bP^2$  branched along the sextic curve $C+ L_1$. The surface $\bar{X}_0$ is a singular $K3$ surface with  five $A_1$ singularities. Let $X_0$ be the $K3$ surface obtained by blowing up the singularities (i.e. take the canonical resolution of $\bar{X}_0$). Denote by $E_1, \cdots, E_5$ the exceptional curves on $X_0$ with self intersection $(-2)$. Set $D_0$ to be the preimage of $L_0$ in $X_0$ and let $D_1\subset X_0$ be the strict transform of $L_1$. Choose $\calL := \calO_{X_0}(D_1 + E_1 + \cdots + E_5)$. By computing the pull-back of $\calO_{\bP^2}(1)$ one sees that that $D_0 \sim 2D_1 + E_1 + \cdots + E_5$ and hence $\calL^{\otimes 2} = \calO_{X_0}(D_0 + E_1 + \cdots + E_5)$. Now we take the branched double cover $S_0$ of $X_0$ along $D_0 + E_1 + \cdots + E_5$. The exceptional curves $E_1, \cdots, E_5$ become $(-1)$-curves on $S_0$. Contract these $(-1)$-curves and one obtains a surface $S$. For later use let us denote by $\sigma_0$ the involution of $S$ so that $\bar{X}_0 = S/\sigma_0$. To summarize we have the following diagram (the left one). 
\begin{equation} \label{diagram}
\begin{tikzcd}
S_0 \arrow{r}{f_0} \arrow{d}{\varphi_0}
               &S \arrow{d}{\psi_0} \arrow[loop right]{}{\sigma_0}\\
X_0 \arrow{r}{g_0} \arrow{d}{\tau_0}
&\bar{X}_0 \arrow{d}{\delta_0} \\
\widetilde{\bP}^2 \arrow{r}{h_0} & \bP^2 
\end{tikzcd}
\hspace{1.5 cm}
\begin{tikzcd}
S \arrow[leftarrow]{r}{f_1} \arrow{d}{\psi_1} \arrow[loop left]{}{\sigma_1}
               &S_1 \arrow{d}{\varphi_1}\\
\bar{X}_1 \arrow[leftarrow]{r}{g_1} \arrow{d}{\delta_1}
&X_1 \arrow{d}{\tau_1} \\
{\bP^2} \arrow[leftarrow]{r}{h_1} & \widetilde{\bP}^2
\end{tikzcd}
\end{equation}

\begin{lemma}
The surface $\psi_0 \circ \delta_0: S \rightarrow \bP^2$ is a Galois cover with group $(\bZ/2\bZ)^2$.
\end{lemma}
\begin{proof}
By a result of Zariski the fundamental group $\pi_1(\bP^2-(C+L_0+L_1))$ is abelian (see for example \cite{Fulton_complement} or \cite[Thm. 1.6]{Catanese_bidouble}). Therefore, the covering map $\psi_0 \circ \delta_0$ is defined by a normal subgroup and is Galois. Clearly, the Galois group is an abelian group of order $4$. But it can not be $\bZ/4\bZ$ because otherwise the branched loci $C+L_0+L_1$ is divisible by $4$ in $\Pic(\bP^2)$ (or one can directly check that $\sigma_0^2=\id$ in the group of deck transformations $\mathrm{Deck}(S/\bP^2)$). 
\end{proof}

Since the Galois group of the bidouble cover $S$ is $(\bZ/2\bZ)^2$, there is a symmetric construction for $S$. Namely, one takes the double cover $\delta_1: \bar{X}_1 \rightarrow \bP^2$ branched along $C+L_0$ and resolves the five $A_1$ singularities to obtain a $K3$ surface $g_1: X_1 \rightarrow \bar{X}_1$. Call the exceptional curves $F_1, \cdots, F_5 \subset X_1$. It can be shown that $(\delta_1\circ g_1)^{-1}(L_1) + F_1 + \cdots + F_5$ is divisible by $2$ in $\Pic(X_1)$. Let $S_1$ be the double cover of $X_1$ along $(\delta_1\circ g_1)^{-1}(L_1) + F_1 + \cdots + F_5$. The surface $S$ is obtained by contracting the $(-1)$-curves on $S_1$. Let us use $\sigma_1$ to denote the involution of $S$ with $\bar{X}_1 = S/\sigma_1$. (Note that $\sigma_0$ and $\sigma_1$ generate the Galois group $(\bZ/2\bZ)^2$.) See the right part of Diagram (\ref{diagram}). 

\begin{proposition} \label{Hodge numbers}
Let $C$ be a smooth quintic curve and $L_0$, $L_1$ be two transverse lines with $C \cap L_0 \cap L_1 = \emptyset$. Let $S$ be the $(\bZ/2\bZ)^2$-cover of $\bP^2$ branched along $C+L_0+L_1$. Then the surface $S$ is a minimal algebraic surface of general type with $p_g(S) =2$ and $K_S^2=1$. Moreover, $S$ is simply connected and the canonical bundle $K_S$ is ample. 
\end{proposition}
\begin{proof}
Notations as above. Denote by $\pi: S \rightarrow \bP^2$ the covering map $\delta_0 \circ \psi_0 = \delta_1 \circ \psi_1$. It is clear from the construction that $S$ is smooth. By \cite[Lem. 3.2]{Morrison_todorov} the canonical bundle $K_S$ of the double cover $S$ can be computed as $2K_S \sim \pi^*\calO_{\bP^2}(1)$. It follows that $K_S$ is ample and $K_S^2=1$. In particular, $K_S$ is big and nef and hence $S$ is a minimal surface of general type. Now let us compute $h^{2,0}(S)$ which clearly equals $h^0(S_0, \calO_{S_0}(K_{S_0}))$. Let $\calL = \calO_{X_0}(D_1 + E_1 + \cdots + E_5)$ be the line bundle associated with the double cover $\varphi_0: S_0 \rightarrow X_0$ of the $K3$ surface $X_0$. Note that $H^0(S_0, \calO_{S_0}(K_{S_0})) = H^0(X_0, \varphi_{0*}\calO_{S_0}(K_{S_0})) = H^0(X_0, \varphi_{0*}\varphi_0^*(\calL))$. Since $\varphi_{0*}\varphi_0^*(\calL) = \calL \otimes \varphi_{0*}(\calO_{X_0}) = \calL \otimes (\calO_{X_0} \oplus \calL^{-1}) = \calL \oplus \calO_{X_0}$, we have $h^0(S_0, \calO_{S_0}(K_{S_0})) = h^0(X_0, \calO_{X_0}) + h^0(X_0, \calL)$. Because $D_1 + E_1 + \cdots + E_5$ is effective and $(D_1 + E_1 + \cdots + E_5)^2 = -2$, the space $H^0(X_0, \calL)$ is $1$-dimensional. Thus $h^{2,0}(S)=h^0(S_0, K_{S_0})=2$ (this is also mentioned in \cite[Rmk. 8]{Catanese_kunev}). By \cite[Thm. 11, Thm. 14]{Bombieri} or \cite[Prop. 2.7]{Catanese_bidouble} the surface $S$ is simply connected.
\end{proof}

Algebraic surfaces of general type with $p_g=2$ and $K^2=1$ have been studied by Horikawa \cite{Horikawa_II}. We call the bidouble covers $S$ constructed above {\sl special Horikawa surfaces}. 

\begin{proposition}
The canonical model of an algebraic surface $Y$ of general type with $q(Y)=0$, $p_g(Y)=2$ and $K_Y^2 = 1$ is a hypersurface of degree $10$ in $\bP(1,1,2,5)$. If $K_Y$ is ample, then $Y$ is isomorphic to a quasi-smooth hypersurface of degree $10$ in $\bP(1,1,2,5)$.
\end{proposition}
\begin{proof}
This has been proved in \cite[\S 2]{Horikawa_II}. See also \cite[\S VII.7]{BPV}. 
\end{proof}

\begin{remark}
A weighted hypersurface $Y \subset \mathbb{P}$ is {\sl quasi-smooth} if the associated affine quasicone is smooth outside the vertex $0$ (cf. \cite{Dolgachev_weighted}). If in addition $\mathrm{codim}_{Y}(Y \cap \bP_{\mathrm{sing}}) \geq 2$, then $Y_{\mathrm{sing}} = Y \cap \bP_{\mathrm{sing}}$. (In our case, we have $\bP = \bP(1,1,2,5)$ which has two singular points $[0,0,1,0]$ and $[0,0,0,1]$ and hence $S_{\mathrm{sing}} = S \cap \bP_{\mathrm{sing}}$.) Moreover, the cohomology $H^k(Y, \bC)$ of a quasi-smooth hypersurface $Y$ admits a pure Hodge structure and explicit calculation can be done using (a slightly generalized version of) Griffiths residue.
\end{remark}

We have shown that special Horikawa surfaces $S$ have ample canonical bundles. A natural question is which degree $10$ quasi-smooth hypersurfaces in $\bP(1,1,2,5)$ do they correspond to.

\begin{proposition} \label{S equation}
Let $S$ be a $(\bZ/2\bZ)^2$-cover of $\bP^2$ branched over a smooth quintic $C$ and two general lines $L_0$, $L_1$ (i.e. $S$ is a special Horikawa surface). Then $S$ is isomorphic to a quasi-smooth hypersurface $z^2 = F(x_0^2, x_1^2, y)$ in $\bP(1,1,2,5) = \mathrm{Proj}(\bC[x_0, x_1, y, z])$ where $F$ is a quintic polynomial.
\end{proposition}
\begin{proof}
Denote by $\pi: S \rightarrow \bP^2$ the covering map $\delta_0 \circ \psi_0 = \delta_1 \circ \psi_1$. For $i=0,1$ we let $\Lambda_i$ be the reduced inverse image $\pi^{-1}(L_i)$ of $L_i$ in $S$. By the proof of Proposition \ref{Hodge numbers} we have $\Lambda_i \in |K_S|$. Choose a section $x_i \in H^0(S, \calO_S(K_S))$ which cuts out $\Lambda_i$. Clearly $\{x_0,x_1\}$ forms a basis of $H^0(S, \calO_S(K_S))$. Since $2K_S \sim \pi^*\calO_{\bP^2}(1)$, the covering map $\pi$ is defined by a subspace of $|2K_S|$. Choose $y \in H^0(S, \calO_S(2K_S))$ so that the subspace of $|2K_S|$ is generated by $x_0^2$, $x_1^2$ and $y$. Note that $y \neq x_0x_1$. By choosing a suitable $z \in H^0(S, \calO_S(5K_S))$ we assume that the equation for $S$ is $z^2 = F'(x_0,x_1,y)$ where $F'$ is a weighted homogeneous polynomial of degree $10$ in $\bP(1,1,2,5)$. (The defining equation must contain $z^2$ otherwise $S$ is not quasi-smooth. Then we complete the square for $z$ which does not affect the other coordinates.) The ramification locus $\pi$ consists of three components $(x_0=0)$, $(x_1=0)$ and $(z=0)$ which are mapped to $L_0$, $L_1$ and $C$ respectively. The proposition then follows.  

Alternatively, one considers the action of $\sigma_i$ (see Diagram (\ref{diagram})) on the canonical ring $\bigoplus_{m\geq0} H^0(S, \calO_S(mK_S))$, hence on $\bP(1,1,2,5)$. Note that the involution $\sigma_0$ fixes $\Lambda_0$ pointwise and $\sigma_0(\Lambda_1) = \Lambda_1$. (Similarly for the involution $\sigma_1$.) As in the proof of \cite[Thm. $3$]{Catanese_kunev} one can choose $y, z$ such that $\sigma_0$ acts on $\bP(1,1,2,5)$ by $[x_0,x_1,y,z] \mapsto [-x_0,x_1,y,z]$. Since $\sigma_0$ acts on $S$ we assume that the defining equation for $S$ is an eigenvector for $\sigma_0$. It is not difficult to see that only even powers of $x_0$ can appear in the defining equation. Next one normalizes the equation and gets $z^2 = F'(x_0,x_1,y)$. Since $F'$ is a weighted homogeneous polynomial of degree $10$, $x_1$ also has even powers.
\end{proof}

\begin{remark}
\leavevmode
\begin{enumerate}
\item Because $S$ is quasi-smooth, the quintic $F$ must contain $y^5$ (more generally, see \cite[Thm. 3.3]{FPR}).  
\item We have shown that a special Horikawa surface $S$ is isomorphic to a quasi-smooth hypersurface in $\bP(1,1,2,5)$ with the equation $z^2 = F(x_0^2, x_1^2, y)$. The covering map $S \rightarrow \bP^2$ is given by $[x_0, x_1, y, z] \mapsto [x_0^2, x_1^2, y]$. The Galois group is generated by $\sigma_{x_0}: [x_0, x_1, y, z] \mapsto [-x_0, x_1, y, z]$ and $\sigma_{x_1}: [x_0, x_1, y, z] \mapsto [x_0, -x_1, y, z]$ which correspond to $\sigma_0$ and $\sigma_1$ respectively.
\item Special Horikawa surfaces have moduli dimension $({2+5 \choose 2}-1) + 2 + 2 - 8 = 16$ (dimension of moduli for plane quintics together with two lines minus dimension of $\PGL(3)$). The dimension of the moduli for degree $10$ quasi-smooth hypersurfaces in $\bP(1,1,2,5) = \mathrm{Proj}(\bC[x_0, x_1, y, z])$ cut out by $z^2 = F(x_0^2, x_1^2, y)$ is ${2+5 \choose 2} - 5 = 16$ (dimension of the quintic polynomials $F$ minus dimension of the subgroup of the automorphism group of $\bP(1,1,2,5)$ acting on $z^2 = F(x_0^2, x_1^2, y)$: a semidirect product of the group consisting of the elements $[x_0, x_1, y, z] \mapsto [ax_0, bx_1, cy+dx_0^2+ex_1^2, z]$ with the group generated by $[x_0, x_1, y, z] \mapsto [x_1, x_0, y, z]$).
\end{enumerate}
\end{remark}

Now let us study how the Galois group $(\bZ/2\bZ)^2$ acts on the Hodge structures of special Horikawa surfaces $S$. We view $S$ as a degree $10$ quasi-smooth hypersurface in $\bP(1,1,2,5)$ cut out by the equation $z^2 = F(x_0^2, x_1^2, y)$. Choose the K\"ahler form corresponding to the canonical curve $(x_0=0)$ or $(x_1=0)$ (which are the reduced inverse image of $L_0$ and $L_1$ respectively) and the primitive cohomology can be described using Griffiths residue.   

\begin{proposition} \label{residue}
Let $Y$ be a quasi-smooth hypersurface of degree $d$ in a weighted projective space $\bP(a_0,a_1,\cdots,a_n)$. That is, $Y$ is given by a weighted homogeneous polynomial $G(z_0, z_1, \cdots, z_n)$ of degree $d$ whose partial derivatives have no common zero other that the origin. Let $$E = \sum \limits_{i=0}^n z_i\dfrac{\partial}{\partial z_i}$$ be the Euler vector field. Let $dV = dz_0 \wedge \cdots \wedge dz_n$ be the Euclidean volume form, and let $\Omega = i(E)dV$ (where $i$ denotes interior multiplication) be the projective volume form (which has degree $a_0 + \cdots + a_n$). Consider expressions of the form $$\Omega(A) = \dfrac{A \cdot \Omega}{G^{q}}$$ where $A$ is a homogeneous polynomial whose degree is such that $\Omega(A)$ is homogeneous of degree $0$. Then the Poincar\'{e} residues of $\Omega(A)$ span $F^{n-q}H_{\mathrm{prim}}^{n-1}(Y, \bC)$ where $F^{\bullet}$ denotes the Hodge filtration. Moreover, the residue lies in $F^{n-q+1}$ if and only if $A$ lies in the Jacobian ideal $J_G$ of $G$ (the ideal generated by the first partial derivatives of $G$).
\end{proposition}
\begin{proof}
This is \cite[Prop. 1.2]{ACT}. See also \cite{Dolgachev_weighted}. 
\end{proof}

\begin{proposition} \label{decomposition}
Let $S \in \bP(1,1,2,5)$ be a quasi-smooth hypersurface of degree $10$ given by $z^2 = F(x_0^2, x_1^2, y)$. Set $\sigma_0$ and $\sigma_1$ to be the automorphisms defined by $[x_0, x_1, y, z] \mapsto [-x_0, x_1, y, z]$ and $[x_0, x_1, y, z] \mapsto [x_0, -x_1, y, z]$ respectively. Let $\chi_0$ and $\chi_1$ be the corresponding characters of the Galois group defined by $\chi_0(\sigma_0) = 1$, $\chi_0(\sigma_1) = -1$ and similarly for $\chi_1$. Then we have the following decomposition of Hodge structures ($H^2_{\chi_l}(S,\bQ)$ is the eigenspace corresponding to $\chi_l$ for $l=0,1$) $$H^2_{\mathrm{prim}}(S,\bQ) = H^2_{\chi_0}(S,\bQ) \oplus H^2_{\chi_1}(S,\bQ).$$ Moreover, $H^2_{\chi_l}(S,\bQ)$ ($l=0,1$) has Hodge numbers $[1,14,1]$.
\end{proposition} 
\begin{proof}
Notations as in Proposition \ref{residue}. The decomposition is obtained by a Griffiths residue calculus. Specifically, let $G= F(x_0^2, x_1^2, y) - z^2$. Take a basis for $H^2_{\mathrm{prim}}(S,\bC)$ of the forms $\Res \frac{A \cdot \Omega}{G^{q}}$ with $A$ being certain monomials $x_0^ix_1^jy^k$. The cohomology group $H^{2,0}(S)$ (resp. $H_{\mathrm{prim}}^{1,1}(S), H^{0,2}(S))$ correspond to $q=1$ (resp. $q=2$, $q=3$) and hence $i+j+2k=1$ (resp. $i+j+2k=11$, $i+j+2k=21$). In particular, $i+j$ must be odd. It follows that only the characters $\chi_0$ and $\chi_1$ appear in the decomposition of $H^2_{\mathrm{prim}}(S,\bC)$ (and hence $H^2_{\mathrm{prim}}(S,\bQ)$). The eigenspace $H^2_{\chi_0}(S,\bQ)$ is a sub-Hodge structure because $H^2_{\chi_0}(S,\bQ) = \ker(\sigma_0^* -\id) (= \ker(\sigma_1^* +\id))$. Similarly for $H^2_{\chi_1}(S,\bQ)$. The claim on Hodge numbers can be checked for a special surface $z^2=x_0^{10} + x_1^{10} + y^5$ (consider the induced action of $\sigma_0$ or $\sigma_1$ on the polarized variation of rational Hodge structure with fibers $H^2_{\mathrm{prim}}(S,\bQ)$). 
\end{proof}

\begin{remark} 
Let $X_0$ and $X_1$ be the $K3$ surfaces associated to a special Horikawa surface $S$ in Diagram (\ref{diagram}). One can show that $T_{\bQ}(S)  \cong T_{\bQ}(X_0) \oplus T_{\bQ}(X_1)$ where $T_\bQ = \mathrm{Tr} \otimes \bQ$ denotes the rational transcendental lattice. Specifically, we consider the induced action $\sigma_0^*$ and $\sigma_1^*$ on $T_{\bQ}(S)$. The space $T_{\bQ}(S)$ decomposes as a direct sum of the eigenspaces of $\sigma_0^*$ with eigenvalues $1$ and $-1$. By \cite[Prop. $5$]{Shioda} we have $T_{\bQ}(X_0) \cong T_{\bQ}(S)^{\sigma_0^*}$ and $T_{\bQ}(X_1) \cong T_{\bQ}(S)^{\sigma_1^*}$ (where $T_{\bQ}(S)^{\sigma_l^*}$ denotes the invariant part). Proposition \ref{decomposition} allows us to identify $T_{\bQ}(S)^{\sigma_1^*}$ with the $(-1)$-eigenspace of $\sigma_0^*$ which completes the proof. Generically, $T_{\bQ}(S)$ has Hodge numbers $[2,28,2]$ (one applies \cite[Prop. 4]{Moonen} to $z^2 = x_0^{10} + x_1^{10}+ y^5$ to see the generic Picard number is $1$) and $T_{\bQ}(X_l)$ has Hodge numbers $[1,14,1]$ (cf. \cite[Cor. 4.15]{Laza_n16}) for $l=0,1$. 
\end{remark}

\section{The infinitesimal Torelli theorem} \label{sec_inftorelli}
We shall show in this section that (unlike for special Kunev surfaces \cite{Kynev} \cite{Catanese_kunev} \cite{Todorov_kunev}) the infinitesimal Torelli holds for special Horikawa surfaces $S$.

\begin{theorem}
Let $S$ be a bidouble cover of $\bP^2$ branched along a smooth quintic $C$ and two transverse lines $L_0$ and $L_1$ with $C \cap L_0 \cap L_1 = \emptyset$. The natural map 
\begin{equation} \label{inf torelli}
p: H^1(S, \calT_S) \rightarrow \Hom(H^0(S, \omega_S), H^1(S, \Omega_S^1)),
\end{equation}
given by cup product, is injective.
\end{theorem}
\begin{proof}
Usui \cite{Usui_wci} has proved the infinitesimal Torelli theorem for the periods of holomorphic $d$-forms on certain $d$-dimensional complete intersections ($d \geq 2$) in certain weighted projective spaces. One can check that the conditions in op. cit. Theorem $2.1$ are satisfied for the special Horikawa surface $S$ and then apply the theorem. Specifically, by \propositionref{S equation} the surface $S \subset \bP(1,1,2,5)$ is defined by $z^2 = F(x_0^2, x_1^2, y)$ and it is not difficult to verify that $z^2 - F(x_0^2, x_1^2, y), x_0, x_1, y$ forms a regular sequence for $\bC[x_0,x_1,y,z]$. (One can apply op. cit. Proposition 3.1 to check the conditions and prove the infinitesimal Torelli for the periods of holomorphic $2$-forms on any smooth hypersurface with ample canonical bundle in $\bP(1,1,2,5)$.)   
\end{proof}

Pardini \cite[Thm 3.1]{Pardini_surface} \cite[Thm 4.2]{Pardini_torelli} has considered the infinitesimal Torelli problem for certain abelian covers (including bidouble covers). The conditions of the theorems do not hold for special Horikawa surfaces $S$ but the strategy still works. (In particular, we need the notation of prolongation bundle discussed in \cite[\S 2]{Pardini_torelli}.) Let us sketch the proof which might be useful for finding more boundary cases of Pardini's theorems. 

\begin{proof}
The idea is to decompose the infinitesimal Torelli map $$p: H^1(S, \calT_S) \rightarrow \Hom(H^0(S, \omega_S), H^1(S, \Omega_S^1))$$ using the Galois group action. The first step is to figure out the building data (see \cite[\S 2]{Catanese_bidouble} and \cite[Def. 2.1]{Pardini_cover}) of the bidouble cover $\pi: S \rightarrow \bP^2$. The (reduced) branched locus of $\pi$  consists of three irreducible components: $$D_0 := L_0, \,\, D_1 := L_1, \,\, \text{and} \,\, D_z := C.$$ Let $\sigma_0$, $\sigma_1$ be the involutions in Diagram (\ref{diagram}). Let $\sigma_z := \sigma_0 \circ \sigma_1$. The Galois group $G$ of the abelian cover $S \rightarrow \bP^2$ consists of $\id$, $\sigma_0$, $\sigma_1$, $\sigma_z$. Let $\chi_0$, $\chi_1$, $\chi_z$ be the corresponding nontrivial characters of $G$ (and we shall denote the character group by $G^*$). Write $$\pi_*\calO_S = \calO_{\bP^2} \oplus \calL_{\chi_0}^{-1} \oplus \calL_{\chi_1}^{-1} \oplus \calL_{\chi_z}^{-1}$$ where $\calL_{\chi}^{-1}$ denotes the eigensheaf on which $G$ acts via the character $\chi$. One easily verify that $\calL_{\chi_0} = \calL_{\chi_1} = \calO_{\bP^2}(3)$ and $\calL_{\chi_z} = \calO_{\bP^2}(1)$. 

Next we compute the direct images of various sheaves. Let $D = D_0 + D_1 + D_z$. For the trivial character $\chi=1$, define $\Delta_{1} = D$. Set $$\Delta_{\chi_0} := D_0, \,\, \Delta_{\chi_1} := D_1, \,\, \Delta_{\chi_z} := D_z.$$ We also let $$D_{1, 1^{-1}} := \emptyset, \,\, D_{\chi_0, \chi_0^{-1}} := D_1 + D_z, \,\, D_{\chi_1, \chi_1^{-1}} := D_0 + D_z, \,\, D_{\chi_z, \chi_z^{-1}} := D_0 + D_1. $$ (For every pairs of characters $\chi, \phi \in G^*$, $D_{\chi,\phi}$ is defined in \cite{Pardini_surface} and \cite{Pardini_torelli}. The fundamental relations of the building data are $\calL_{\chi} + \calL_{\phi} \equiv \calL_{\chi\phi} + D_{\chi,\phi}$, see \cite[Thm. 2.1]{Pardini_cover}.) For any character $\chi$, we have (see \cite[Prop. 4.1]{Pardini_cover})
\begin{itemize}
\item $(\pi_*\calT_S)^{(\chi)} = \calT_{\bP^2}(-\log \Delta_{\chi}) \otimes \calL_{\chi}^{-1}$ (In particular, $(\pi_*\calT_S)^{(\mathrm{inv})} = \calT_{\bP^2}(-\log D)$);
\item $(\pi_*\Omega^1_S)^{(\chi)} = \Omega_{\bP^2}^1(\log D_{\chi, \chi^{-1}}) \otimes \calL_{\chi}^{-1}$ (In particular, $(\pi_*\Omega^1_S)^{(\mathrm{inv})} = \Omega^1_{\bP^2}$);
\item $(\pi_*\omega_S)^{(\chi)} = \omega_{\bP^2}\otimes \calL_{\chi^{-1}}$ (In particular, $(\pi_*\omega_S)^{(\mathrm{inv})} = \omega_{\bP^2}$.
\end{itemize}

Since the map $\pi$ is finite, for every coherent sheaf $\calF$ on $S$ one has $H^k(S, \calF) = H^k(\bP^2, \pi_*\calF)$ ($k=0,1,2$). In particular, we have $$H^1(S, \calT_S) = H^1(\bP^2, \pi_*\calT_S), \,\, H^1(S, \Omega_S) = H^1(\bP^2, \pi_*\Omega_S), \,\, H^0(S, \omega_S) = H^0(\bP^2, \pi_*\omega_S).$$ Combining with the splittings of $\pi_*\calT_S$, $\pi_*\Omega_S$ and $\pi_*\omega_S$, we obtain the following decompositions:
\begin{equation}
H^1(S, \calT_S) = H^1(\bP^2, \calT_{\bP^2}(-\log D)) \oplus (\mathop{\oplus} \limits_{\chi \in G^* \backslash \{1\}} H^1(\bP^2, \calT_{\bP^2}(-\log \Delta_{\chi}) \otimes \calL_{\chi}^{-1}))
\end{equation}
\begin{equation}
H^1(S, \Omega^1_S) = H^1(\bP^2, \Omega^1_{\bP^2}) \oplus (\mathop{\oplus} \limits_{\chi \in G^* \backslash \{1\}} H^1(\bP^2, \Omega^1_{\bP^2}(\log D_{\chi, \chi^{-1}}) \otimes \calL_{\chi}^{-1}))
\end{equation}
\begin{equation}
H^0(S, \omega_S) = H^0(\bP^2, \omega_{\bP^2}) \oplus (\mathop{\oplus} \limits_{\chi \in G^* \backslash \{1\}} H^0(\bP^2, \omega_{\bP^2} \otimes \calL_{\chi^{-1}}))
\end{equation}

Since the cup product (and hence the infinitesimal Torelli map $p: H^1(S, \calT_S) \rightarrow \Hom(H^0(S, \omega_S), H^1(S, \Omega_S^1))$) is compatible with the group action, for characters $\chi, \phi \in G^*$ we consider 
\begin{equation}
\begin{split}
p_{\chi,\phi}:  & H^1(\bP^2, \calT_{\bP^2}(-\log \Delta_{\chi}) \otimes \calL_{\chi}^{-1})) \\
& \rightarrow \Hom(H^0(\bP^2, \omega_{\bP^2} \otimes \calL_{\phi^{-1}})
,H^1(\bP^2, \Omega^1_{\bP^2}(\log D_{\chi\phi, (\chi\phi)^{-1}}) \otimes \calL_{\chi\phi}^{-1})))
\end{split}
\end{equation}
Clearly, one has 
\begin{equation}
p = \mathop{\oplus} \limits_{\chi, \phi \in G^*} p_{\chi, \phi}.
\end{equation}
\begin{lemma} \label{p_{chi, phi}}
The infinitesimal Torelli holds for $S$ (i.e. the map $p$ is injective) if and only if $\mathop{\cap} \limits_{\phi \in G^*} \ker p_{\chi, \phi} = \{0\}$ for any character $\chi$.
\end{lemma}

Let us take a closer look at the maps $p_{\chi,\phi}$. For every pairs of characters $\chi, \phi \in G^*$ and every section $\xi \in H^0(\bP^2, \omega_{\bP^2}\otimes \calL_{\phi^{-1}})$, consider the following diagram
\begin{equation*}
\begin{CD}
\calT_{\bP^2}(-\log \Delta_{\chi}) \otimes \calL_{\chi}^{-1}     @>>>  \Omega^1_{\bP^2}(\log D_{\chi\phi, (\chi\phi)^{-1}}) \otimes \omega_{\bP^2}^{-1} \otimes (\calL_{\chi\phi} \otimes \calL_{\phi^{-1}})^{-1}\\
@VVV        @VVV\\
\calT_{\bP^2}(-\log \Delta_{\chi}) \otimes \calL_{\chi}^{-1} \otimes \omega_{\bP^2}\otimes \calL_{\phi^{-1}}     @>>>  \Omega^1_{\bP^2}(\log D_{\chi\phi, (\chi\phi)^{-1}}) \otimes \calL_{\chi\phi}^{-1}
\end{CD}
\end{equation*} 
where the vertical maps are given by multiplication by $\xi$ and the horizontal maps are defined by contraction of tensors and by the fundamental relations \cite[Thm. 2.1]{Pardini_cover} of the building data (N.B. there is a canonical isomorphism between $\calT_{\bP^2}(-\log D_{\chi\phi, (\chi\phi)^{-1}})$ and $\Omega^1_{\bP^2}(\log D_{\chi\phi, (\chi\phi)^{-1}}) \otimes (\omega_{\bP^2}(D_{\chi\phi, (\chi\phi)^{-1}}))^{-1}$). Consider
\begin{equation}
\begin{split}
q_{\chi,\phi}: &H^1(\bP^2, \calT_{\bP^2}(-\log \Delta_{\chi}) \otimes \calL_{\chi}^{-1}) \\
& \rightarrow H^1(\bP^2, \Omega^1_{\bP^2}(\log D_{\chi\phi, (\chi\phi)^{-1}}) \otimes \omega_{\bP^2}^{-1} \otimes (\calL_{\chi\phi} \otimes \calL_{\phi^{-1}})^{-1})
\end{split}
\end{equation}
and 
\begin{equation}
\begin{split}
r_{\chi,\phi}: &H^1(\bP^2, \Omega^1_{\bP^2}(\log D_{\chi, \chi^{-1}}) \otimes (\omega_{\bP^2} \otimes \calL_{\chi} \otimes \calL_{\phi^{-1}})^{-1}) \\
& \rightarrow \Hom(H^0(\bP^2, \omega_{\bP^2} \otimes \calL_{\phi^{-1}})
,H^1(\bP^2, \Omega^1_{\bP^2}(\log D_{\chi, \chi^{-1}}) \otimes \calL_{\chi}^{-1})))
\end{split}
\end{equation}
Obviously we have $p_{\chi,\phi} = r_{\chi\phi,\phi} \circ q_{\chi, \phi}$. 

Let us analyze the maps $r_{\chi,\phi}$.
\begin{itemize}
\item The maps $r_{1,1}$, $r_{1,\chi_0}$, $r_{1,\chi_1}$, $r_{1,\chi_z}$ are injective: by explicit computation and Bott's vanishing theorem. 
\item The maps $r_{\chi_0,1}$, $r_{\chi_1,1}$, $r_{\chi_z,1}$, $r_{\chi_0,\chi_z}$, $r_{\chi_1,\chi_z}$ are zero maps: in these cases $H^0(\bP^2, \omega_{\bP^2} \otimes \calL_{\phi^{-1}}) = 0$.
\item The maps $r_{\chi_0,\chi_0}$, $r_{\chi_0,\chi_1}$, $r_{\chi_1,\chi_0}$, $r_{\chi_1,\chi_1}$ are injective: we use the prolongation bundle of the irreducible components of $D_{\chi,\chi^{-1}}$ defined in \cite[\S 2]{Pardini_torelli} and show that if the multiplication map \begin{equation*}
\begin{split}
& H^0(\bP^2, \omega_{\bP^2} \otimes \calL_{\phi^{-1}}) \otimes (\mathop{\oplus} \limits_{
\substack{
B \, \text{irreducible} \\
\text{components of} \, D_{\chi,\chi^{-1}}}} H^0(\bP^2, \calO_{\bP^2}(B) \otimes\omega_{\bP^2} \otimes \calL_{\chi})) \\
& \rightarrow \mathop{\oplus} \limits_{
\substack{
B \, \text{irreducible} \\
\text{components of} \, D_{\chi,\chi^{-1}}}} H^0(\bP^2, \calO_{\bP^2}(B) \otimes \omega_{\bP^2}^{\otimes 2} \otimes \calL_{\chi} \otimes \calL_{\phi^{-1}})
\end{split}
\end{equation*}
is surjective, then the map $r_{\chi,\phi}$ is injective (cf. \cite[\S 3]{Pardini_torelli}). As argued in \cite[Prop. 3.5]{Pardini_torelli} the surjectivity of the multiplication map follows from a special case of \cite[Thm. 2.1]{EL}.
\item The maps $r_{\chi_z,\chi_0}$, $r_{\chi_z,\chi_1}$, $r_{\chi_z,\chi_z}$ are injective: we also consider the prolongation bundle, and also note that $H^0(\bP^2, \calO_{\bP^2}(B) \otimes \omega_{\bP^2}^{\otimes 2} \otimes \calL_{\chi} \otimes \calL_{\phi^{-1}}) = 0$ (with $B$ any irreducible component of $D_{\chi,\chi^{-1}}$) in these cases.  
\end{itemize}

We use the method of \cite[\S 3]{Pardini_surface} to study the maps $q_{\chi,\phi}$ (especially Diagram $(3.4)$ and Lemma $3.1$). The results are as follows.
\begin{itemize}
\item For $\chi = 1$, we have $\ker q_{1,\chi_0} \cap \ker q_{1, \chi_1} = \{0\}$ (and hence $\ker p_{1,\chi_0} \cap \ker p_{1, \chi_1} = \{0\}$).
\item For $\chi = \chi_0$, both $q_{\chi_0,\chi_1}$ and $q_{\chi_0,\chi_z}$ are injective (and hence $p_{\chi_0,\chi_1} = r_{\chi_z,\chi_1} \circ q_{\chi_0, \chi_1}$ is injective).
\item For $\chi = \chi_1$, both $q_{\chi_1,\chi_0}$ and $q_{\chi_1,\chi_z}$ are injective (and hence $p_{\chi_1,\chi_0} = r_{\chi_z,\chi_0} \circ q_{\chi_1, \chi_0}$ is injective).
\item For $\chi = \chi_z$, both $q_{\chi_z,\chi_0}$ and $q_{\chi_z,\chi_1}$ are injective (and hence $p_{\chi_z,\chi_0} = r_{\chi_1,\chi_0} \circ q_{\chi_z, \chi_0}$ is injective, and $p_{\chi_z,\chi_1} = r_{\chi_0,\chi_1} \circ q_{\chi_z, \chi_1}$ is also injective).
\end{itemize}

The theorem clearly follows from these observations.
\end{proof}

\section{Degree $5$ pairs and a generic global Torelli theorem} \label{sec_globaltorelli}
Let us review the period map for degree $5$ pairs which will be used later in this section to prove a generic global Torelli problem for special Horikawa surfaces $S$.  

Following \cite[Def. 2.1]{Laza_n16} we call a pair $(C,L)$ consisting of a plane quintic curve $C$ and a line $L \subset \bP^2$ a {\sl degree $5$ pair}. Two such pairs are equivalent if they are projectively equivalent. We are interested in the degree $5$ pairs $(C, L)$ with $C+L$ defining a sextic curve admitting at worst $ADE$ singularities. The coarse moduli space $\calM_{ADE}$ is contained in the GIT quotient $(\bP H^0(\bP^2, \calO_{\bP^2}(5)) \times \bP H^0(\bP^2, \calO_{\bP^2}(1))) \git \SL_3(\bC)$ (with respect to the linearization $\pi_1^*\calO_{\bP^2}(1) \otimes \pi_2^*\calO_{\bP^2}(1)$).  
 
To every degree $5$ pair $(C,L)$ such that $C+L$ has at worst $ADE$ singularities, we associate a $K3$ surface $X_{(C,L)}$ obtained by taking the canonical resolution of the double cover $\bar{X}_{(C,L)}$ of $\bP^2$ along the sextic $C+L$. The period map for degree $5$ pairs $(C,L)$ is defined via the periods of $X_{(C,L)}$. Specifically, one first considers a generic pair $(C,L)$ where $C$ is smooth and $L$ is a transversal. The $K3$ surface $X_{(C,L)}$ contains $5$ exceptional curves $e_1, \cdots, e_5$ corresponding to the five intersection points $C \cap L$ and the strict transform $l'$ of $L$. By \cite[Prop. 4.12, Def. 4.17]{Laza_n16}, the lattice generated by $\{l', e_1, \cdots, e_5\}$ is a primitive sublattice of $\Pic(X_{(C,L)})$ and hence $X_{(C,L)}$ is a lattice polarized $K3$ surface (note that the lattice polarization depends on the labeling of the points of intersection $C \cap L$).  

\begin{notation}
Let $\Lambda$ be an even lattice. We define:
\begin{itemize}
\item $\Lambda^*$: the dual lattice;
\item $A_{\Lambda} = \Lambda^*/\Lambda$: the discriminant group endowed with the induced quadratic form $q_{\Lambda}$;
\item $O(\Lambda)$: the group of isometries of $\Lambda$;
\item $O(q_\Lambda)$: the automorphisms of $A_\Lambda$ that preserves the quadratic form $q_{\Lambda}$;
\item $O_{-}(\Lambda)$: the subgroup of isometries of $\Lambda$ of spinor norm $1$ (see also \cite[\S 3.6]{Scattone_bb});
\item $\widetilde{O}(\Lambda) =\ker(O(\Lambda) \rightarrow O(A_{\Lambda}))$: the group of isometries of $\Lambda$ that induce the identity on $A_{\Lambda}$;
\item $O^*(\Lambda): = O_-(\Lambda) \cap \widetilde{O}(\Lambda)$.
\end{itemize}
We also introduce:
\begin{itemize}
\item $\Lambda_{K3}$: the $K3$ lattice $U^{\oplus 3} \oplus E_8^{\oplus 2}$ (we denote the bilinear form by $(\cdot,\cdot)_{K3}$);
\item $M$: the abstract lattice generated by $\{l', e_1, \cdots, e_5\}$ which admits a unique primitive embedding into $\Lambda_{K3}$ (one can also show that $M$ is the generic Picard group of $K3$ surfaces $X_{(C,L)}$ and $M \cong U(2) \oplus D_4$, see \cite[Cor. 4.15, Lem. 4.18]{Laza_n16}); 
\item $T = M^{\perp}_{\Lambda_{K3}}$: the orthogonal complement of $M$ in $\Lambda_{K3}$ (which is isomorphic to $U \oplus U(2) \oplus D_4 \oplus E_8$);
\item $\calD_M := \{\omega \in \bP(T\otimes \bC) \mid (\omega, \omega)_{K3} = 0, (\omega, \bar{\omega})_{K3} > 0\}$ which is the period domain for $M$-polarized $K3$ surfaces;
\item $\calD_M^0:$ a connected component of $\calD_M$ which is a type IV Hermitian symmetric domain.
\end{itemize}
\end{notation}

Let $\calU$ be an open subset of $\calM_{ADE}$ parameterizing the generic degree $5$ pairs $(C, L)$ with $C$ smooth and with transverse intersections $C \cap L$. Let $\widetilde{\calU}$ be the $\mathfrak{S}_5$-cover of $\calU$ that consists of triples $(C,L,\sigma)$ with $\sigma: \{1,\cdots, 5\} \rightarrow C \cap L$ labelings of $C \cap L$. By \cite{Dolgachev_latticek3} there is a period map $\widetilde{\calU} \rightarrow \calD_M^0/O^*(T)$ sending $(C,L,\sigma)$ to the periods of $X_{(C,L)}$ with the $M$-polarization determined by $\sigma$. By the global Torelli theorem and surjectivity of the period map for $K3$ surfaces and \cite[Prop. 4.14]{Laza_n16} the period map is birational. Note that there is a natural $\mathfrak{S}_5$-action on $\widetilde{\calU}$. Moreover, the group $O^*(T)$ is a normal subgroup of $O_-(T)$ with $O_-(T)/O^*(T) \cong \mathfrak{S}_5$, and the residual $\mathfrak{S}_5$-action on $M$ is the permutation of the five points of intersection $C \cap L$ (op. cit. Proposition 4.22). (In fact, $O_-(T)$ is the monodromy group for the degree $5$ pairs.) Thus, the period map is $\mathfrak{S}_5$-equivariant and descends to a birational map $\calU \rightarrow \calD_M^0/O_-(T)$.

The birational map can be extended to a morphism $\calM_{ADE} \rightarrow \calD_M^0/O_-(T)$ using normalized $M$-polarizations. In particular, one needs to construct an $M$-polarization in the case of non-transversal intersections $C \cap L$. We briefly summarize the construction and refer the readers to \cite[\S4.2.3]{Laza_n16} for the details. The construction is a modification of canonical resolution of  singularities of double covers (see \cite[Thm. III.7.2]{BPV}). The role of modification is to keep track of the points of intersection $C \cap L$. More precisely, one chooses a labeling of the intersection $\sigma: \{1,2,3,4,5\} \twoheadrightarrow C \cap L$ such that for any $p \in C \cap L$ we have $|\sigma^{-1}(p)| = \mathrm{mult}_p(C \cap L)$. Set $Y_0 = \bP^2$ and $B_0 = C+L$. We blow up one singularity at a time (instead of doing simultaneous blow-ups) and do the first five blow-ups in points belonging to $L$. The new branched divisor $B_{i}$ is the strict transform of $B_{i-1}$ together with the exceptional divisor of the blow-up reduced mod $2$. The process is repeated until the resulting divisor $B_N$ is smooth. Denote the blow-up sequence by $Y_N \rightarrow \cdots \rightarrow Y_{i} \rightarrow Y_{i-1} \rightarrow \cdots \rightarrow Y_0=\bP^2$. The double cover $X_{(C,L)}$ of $Y_N$ along $B_N$ is a minimal resolution of $\bar{X}_{(C,L)}$. Let $p_i \in Y_{i-1}$ ($1 \leq i \leq 5$) be the centers of the blow-up which lies on the corresponding strict transform of $L$. Now we construct a primitive embedding of $M$ into $\Pic(X_{(C,L)})$ by sending $l$ to the class of the reduced preimage of $L$ and sending $e_i$ ($1 \leq i \leq 5$) to the fundamental cycle associated to the simple singularity of $B_{i-1}$ in the point $p_i$. The embedding is normalized in the sense of \cite[Def. 4.24]{Laza_n16} and the construction fits well in families. 

By the global Torelli theorem for $K3$, the surface $X_{(C,L)}$ is unique up to isomorphism. Moreover, one can recover the degree $5$ pair $(C,L)$ because the classes $2l'+e_1+\cdots+e_5$ (which corresponds to the pull-back of $\calO_{\bP^2}(1)$ and determines the covering map and the branched curve) and $l'$ (which determines the line $L$ and hence the residue quintic $C$) are fixed by the monodromy group $O_-(T)$. It follows that the period map $$\calM_{ADE} \hookrightarrow \calD_M^0/O_-(T)$$ for degree $5$ pairs $(C, L)$ with $C+L$ admitting at worst $ADE$ singularities is injective. This is the part we shall need later. For the completeness, let us mention that one can verify that the period map is surjective (see \cite[\S4.3.1]{Laza_n16}, especially Proposition 4.31). By Zariski's main theorem, the bijective birational morphism between two normal varieties $\calM_{ADE} \rightarrow \calD^0_M/O_-(T)$ is an isomorphism (op. cit. Theorem 4.1). 

Let us focus on the generic global Torelli problem for special Horikawa surfaces. By \cite{Catanese_bidouble} or \cite[Thm. 2.1]{Pardini_cover} we construct the coarse moduli space $\calM$ for special Horikawa surfaces as the open subset of the quotient\footnote{for the linearization induced by $\pi_1^*\calO_{\bP^2}(1) \otimes \pi_2^*\calO_{\bP^2}(1) \otimes \pi_3^*\calO_{\bP^2}(1)$} $$(\bP H^0(\bP^2, \calO_{\bP^2}(5)) \times \Sym^2 (\bP H^0(\bP^2, \calO_{\bP^2}(1)))) \git \SL_3(\bC)$$ corresponding to triples $(C, L, L')$ which consist of smooth quintics $C$ and transversals $L$ and $L'$ with $C \cap L \cap L' = \emptyset$. It is more convenient to work with a double cover $\calM'$ of $\calM$. Specifically, $\calM'$ is the open subset of the GIT\footnote{with respect to the linearization $\pi_1^*\calO_{\bP^2}(1) \otimes \pi_2^*\calO_{\bP^2}(1) \otimes \pi_3^*\calO_{\bP^2}(1)$} $$(\bP H^0(\bP^2, \calO_{\bP^2}(5)) \times \bP H^0(\bP^2, \calO_{\bP^2}(1)) \times \bP H^0(\bP^2, \calO_{\bP^2}(1))) \git \SL_3(\bC)$$ which parameterizes (up to projective equivalence) triples $(C, L_0, L_1)$ with $C$ smooth quintics and $L_0$, $L_1$ ``labeled" lines which intersect $C$ transversely and satisfy $C \cap L_0 \cap L_1 = \emptyset$.

Choose a sufficiently general reference point $b \in \calM'$ (in particular, we label the two lines) and let $S_b$ be the corresponding bidouble cover. Let $V = H_{\mathrm{prim}}^2(S_b, \bR)$ (with respect to the class of a canonical curve or equivalently a hyperplane section in $\bP(1,1,2,5)$). Let $Q$ be the polarization on $V$ defined using cup product. We also write $V_{\bQ} = H_{\mathrm{prim}}^2(S_b, \bQ)$. Consider the action of the Galois group $(\bZ/2\bZ)^2$ on $S_b$ and define $\rho: (\bZ/2\bZ)^2 \rightarrow \Aut(V, Q)$ to be the corresponding representation. Notations as in Diagram (\ref{diagram}). The Galois group $(\bZ/2\bZ)^2$ is generated by $\sigma_0$ and $\sigma_1$. Denote the corresponding characters by $\chi_0$ and $\chi_1$. 

\begin{notation}
We shall use the following notations:
\begin{itemize}
\item $\calD = \calD(V, Q)$: the period domain parameterizing $Q$-polarized Hodge structures of weight $2$ on $V$ with hodge numbers $[2,28,2]$;
\item $\calD^{\rho} = \{x \in \calD \mid \rho(a)(x) = x, \forall a \in  (\bZ/2\bZ)^2\}$;
\item $V(\chi)$: the eigenspace of $V$ corresponding to the character $\chi$ (it is not difficult to see that the eigenspaces $V(\chi)$ and $V(\chi')$ are orthogonal with respect to $Q$ if $\chi \neq \chi'$);
\item $V(\chi)_{\bQ} := V(\chi) \cap V_{\bQ}$; 
\item $\calD(\chi)$: the period domain $\calD(V(\chi), Q|_{V(\chi)})$ of type $[1,14,1]$. 
\end{itemize}
\end{notation}

\begin{lemma}
There is a natural map $\calD^{\rho} \rightarrow \calD(\chi_0) \times \calD(\chi_1)$ which is injective.
\end{lemma}
\begin{proof}
The lemma follows from \propositionref{decomposition} and \cite[\S7]{DK_ball}. Specifically, only the characters $\chi_0$ and $\chi_1$ appear in the decomposition of the vector space $V$. Let $V \otimes \bC = V^{2,0} \oplus V^{1,1} \oplus V^{0,2}$ be a $Q$-polarized Hodge structure on $V$. The map is defined by sending the Hodge structure to the induced $Q$-polarized Hodge structures on $V(\chi_0)$ and $V(\chi_1)$ which is clearly injective. 
\end{proof}

Now we show that the period spaces $\calD(\chi_0)$ and $\calD(\chi_1)$ are both isomorphic to the period space $\calD_M$ for $M$-polarized $K3$ surfaces.

\begin{lemma} \label{reference point}
There exists an isomorphism ($l=0 \,\, \text{or}\,\,1$) $$(V(\chi_l)_{\bQ}, \frac12 Q) \cong (T \otimes \bQ, (\cdot,\cdot)_{K3} \otimes \bQ).$$
\end{lemma}
\begin{proof}
We take $\chi_0$ as an example. Let $S= S_b$ be the bidouble cover corresponding to the reference point $b \in \calM'$. Notations as in Diagram (\ref{diagram}). In particular, by abuse of notation $\sigma_0$ also denotes the involution relative to $\varphi_0: S_0 \rightarrow X_0$. Label the points of intersection $C \cap L_1$ (the isomorphism we shall describe does not depend on the labeling) and there is a primitive embedding of $M$ (and hence $T$) into $H^2(X_0,\bZ)$. Since $b$ is sufficiently general, $M$ is the Picard lattice of the $K3$ surface $X_0$ (\cite[Cor. 4.15]{Laza_n16}). Consider the composition of linear maps $$T \otimes \bQ \hookrightarrow H^2(X_0,\bQ) \stackrel{\varphi_0^*}{\rightarrow} H^2(S_0,\bQ)^{\sigma_0^*}.$$ The map $\varphi_0^*$ is an isomorphism of vector spaces. Let $D_1 \subset X_0$ be the strict transform of $L_1$ and set $E_1, \cdots, E_5$ to be the exceptional curves. Clearly, $T \otimes \bQ$ is the orthogonal complement of $\bQ[D_1] \oplus \bQ[E_1] \oplus \cdots \oplus \bQ[E_5]$ in $H^2(X_0,\bQ)$. Thus, $T \otimes \bQ$ is mapped onto $V(\chi_0)_\bQ = H^2_{\mathrm{prim}}(S, \bQ)^{\sigma_0^*}$ by $\varphi_0^*$. The claim on the bilinear forms is clear.
\end{proof}

\begin{corollary}
The period domain $\calD(\chi_l)$ has two connected components which are both isomorphic to the $14$-dimensional type IV Hermitian symmetric domain $\SO(2,14)/\mathrm{S}(\mathrm{O}(2) \times \mathrm{O}(14))$.
\end{corollary}

To formulate the theorem we also need to choose a discrete group. Let $\Gamma_0$ (resp. $\Gamma_1$) be the discrete subgroup of $\Aut(V(\chi_0)_{\bQ}, Q)$ (resp. $\Aut(V(\chi_1)_{\bQ}, Q)$) corresponding to $O_{-}(T)$ (using \lemmaref{reference point}). Set $\Gamma$ to be the discrete subgroup in $\Aut(V_{\bQ}, Q)$ which projects onto $\Gamma_0$ and $\Gamma_1$ under the isomorphism $V_{\bQ} = V(\chi_0)_{\bQ} \oplus V(\chi_1)_{\bQ}$. Now we consider the period map $$\calP: \calM \rightarrow \calD^{\rho}/\Gamma$$ for special Horikawa surfaces (which are canonically polarized). Because the monodromy group (for the very general base point) is contained in the generic Mumford-Tate group and $\sigma_0^*$ and $\sigma_1^*$ are Hodge tensors for every member of the family, the monodromy representation commutes with the representation $\rho$ (see also \cite[pp.67-68]{GGK_mt} and \cite[\S 7]{DK_ball}). By \cite[Prop. 4.22]{Laza_n16} the discrete subgroup $\Gamma$ contains the image of the monodromy representation. 

To prove the generic global Torelli theorem for $\calP: \calM \rightarrow \calD^{\rho}/\Gamma$ we consider the map $$\calP_0 \times \calP_1: \calM' \rightarrow \calD^0_M/O_{-}(T) \times \calD^0_M/O_{-}(T)$$ which is defined using the period maps for the degree $5$ pairs $(C, L_0)$ and $(C, L_1)$.
\begin{proposition} \label{P0P1}
The map $\calP_0 \times \calP_1: \calM' \rightarrow \calD^0_M/O_{-}(T) \times \calD^0_M/O_{-}(T)$ is generically injective.
\end{proposition}
\begin{proof}
By \cite{Laza_n16} Section 4.2.3 or Theorem 4.1, one can recover a degree $5$ pair $(C, L)$ (up to projective equivalence) from the periods of the $K3$ surface $X_{(C,L)}$. As a result, the isomorphism class of a triple $(C,L_0,L_1)$ is determined by the periods of the $K3$ surfaces $X_{(C,L_0)}$ and $X_{(C,L_1)}$ provided that the quintic $C$ has no nontrivial automorphism. More specifically, we assume that $(\calP_0 \times \calP_1)(C,L_0,L_1) = (\calP_0 \times \calP_1)(C',L'_0,L'_1)$. Then there exist $f,g \in \PGL_3(\bC)$ such that $f(C)=C'$, $f(L_0) = L_0'$, $g(C)=C'$, and $g(L_1) = L_1'$. In particular, one has $(g^{-1} \circ f)(C) = C$. Because $\Aut(C) = \{\id\}$ we get $f=g$. Thus, the map $\calP_0 \times \calP_1: \calM' \rightarrow \calD^0_M/O_{-}(T) \times \calD^0_M/O_{-}(T)$ is generically injective. 
\end{proof} 

\begin{theorem}
The period map $\calP: \calM \rightarrow \calD^{\rho}/\Gamma$ is generically injective.
\end{theorem}
\begin{proof}
Let $\calP': \calM' \rightarrow \calD^{\rho,0}/\Gamma \hookrightarrow \calD^{0}(\chi_0)/\Gamma_0 \times \calD^{0}(\chi_1)/\Gamma_1$ (where the superscript $^0$ denotes the choice of a connected component) be the map sending a labeled triple $(C,L_0,L_1)$ to the periods of the corresponding special Horikawa surface $S$ (and to the eigenperiods on the underlying eigenspaces $V(\chi_0)$ and $V(\chi_1)$). We would like to compare $\calP'$ and $\calP_0 \times \calP_1$. Note that the transcendental lattice $T$ has been identified with the invariant parts of the underlying vector space $V_\bQ = H^2_{\mathrm{prim}}(S_b, \bQ)$ for the involutions $\sigma_l^*$ ($l=0,1$) via the natural pull-backs (see \lemmaref{reference point}). We claim that under these identifications $\calP'$ coincides with $\calP_0 \times \calP_1$ up to the order of the periods (after relabeling the two lines one needs to switch the periods of the two $K3$ surfaces, but the period of the special Horikawa surface and the eigenperiods remain the same). The reason is that the eigenperiods are obtained by pulling back the holomorphic $2$-forms of the two $K3$ surfaces $X_0 = X_{(C,L_1)}$ and $X_1 = X_{(C,L_0)}$ (see Diagram (\ref{diagram}) and \propositionref{decomposition}). Specifically, we choose the orderings for the intersection points $C \cap L_1$ (resp. $C \cap L_0$) and define a primitive embedding of $M$ (and also $T$) into $H^2(X_0, \bZ)$ (resp. $H^2(X_1, \bZ)$) accordingly. Let $\phi_l: \Lambda_{K3} \stackrel{\cong}{\rightarrow} H^2(X_l, \bZ)$ ($l=0,1$) be the markings compatible with the primitive embeddings of $M$. The map $\calP_0 \times \calP_1$ is defined by considering (the $O_-(T)$-orbits of) the $K3$ periods $\phi_{0, \bC}^{-1}(H^{2,0}(X_0))$ and $\phi_{1, \bC}^{-1}(H^{2,0}(X_1))$. Now let us discuss the (eigen)period map $\calP'$. We denote $\sigma_0$ (resp. $\sigma_1$) in Diagram (\ref{diagram}) by $\sigma$ (resp. $\sigma'$) to reflect the fact that they do not depend on how one labels the lines. Let $\chi$ and $\chi'$ be the corresponding characters. We now construct a natural $\rho$-marking (cf. \cite[\S 7]{DK_ball}) on $H_{\mathrm{prim}}^2(S, \bQ)$ using the $K3$ markings $\phi_0$ and $\phi_1$. On one hand, \lemmaref{reference point} gives us $V(\chi)_{\bQ} \cong T \otimes \bQ$ and $V(\chi')_{\bQ} \cong T \otimes \bQ$. On the other hand, one has $T \otimes \bQ \stackrel{\phi_{0}}{\rightarrow} H^2(X_0, \bQ) \stackrel{\cong}{\rightarrow} H^2(S_0, \bQ)^{\sigma^*}$ where the second homomorphism is induced by the natural pull-back map. As in the proof of \lemmaref{reference point} one can show that the image of the composition (which is clearly injective) is $H^2_{\mathrm{prim}}(S, \bQ)^{\sigma^*}$. In other words, we get $T \otimes \bQ \stackrel{\cong}{\rightarrow} H^2_{\mathrm{prim}}(S, \bQ)^{\sigma^*}$. Similarly, the $K3$ marking $\phi_1$ and the natural pull-back map allows us to identify $T \otimes \bQ$ and $H^2_{\mathrm{prim}}(S, \bQ)^{\sigma'^*}$. Combining these observations, we get the markings $V(\chi)_\bQ \stackrel{\cong}{\rightarrow} H^2(S_0, \bQ)^{\sigma^*}$ and $V(\chi')_\bQ \stackrel{\cong}{\rightarrow} H^2(S_0, \bQ)^{\sigma'^*}$. Taking the preimage of the holomorphic $2$-form invariant for $\sigma^*$ (resp. $\sigma'^*$) in $V(\chi)_\bQ \otimes \bC$ (resp. $V(\chi')_\bQ \otimes \bC$) one obtains the eigenperiod map $\calP'$. Now our claim is clear. By \propositionref{P0P1} generically $\calP'$ has degree $2$ (depending on the labelings of the lines). As a result, the period map $\calP: \calM \rightarrow \calD^{\rho}/\Gamma$ is generically injective. 
\end{proof}

\begin{remark}
Assume the quintic curve $C$ admits a nontrivial automorphism $\sigma$ satisfying $\sigma(L_0) \neq L_0$ and $\sigma(L_1) \neq L_1$. Then the triples $(C,L_0,L_1)$ and $(C, \sigma(L_0), L_1)$ are mapped by the period map $\calP$ to the same point. 
\end{remark}

\begin{remark}
Let $\mathcal W$ be the subset of $|\calO_{\bP^2}(5)| \times |\calO_{\bP^2}(1)| \times |\calO_{\bP^2}(1)|$ corresponding to triples $(C, L_0, L_1)$ with $C+L_0$ and $C+L_1$ admitting at worst $ADE$ singularities and $C \cap L_0 \cap L_1 = \emptyset$. By taking bidouble covers we obtain a family of surfaces $\mathcal S \rightarrow \mathcal W$ with only du Val singularities. By applying a simultaneous resolution to the family $\mathcal S$ we obtain a family $\widetilde{\mathcal S} \rightarrow\mathcal W$ (after a finite base change of $\mathcal W$) of Horikawa surfaces (which are surfaces of general type with $p_g=2$ and $K^2=1$). Consider the period map $\calP_0 \times \calP_1: \mathcal W \rightarrow \calD^0_M/O_{-}(T) \times \calD^0_M/O_{-}(T)$. The generic global Torelli theorem holds for this family. Namely, if two generic points in $\mathcal W$ have the same image in $\calD^0_M/O_{-}(T) \times \calD^0_M/O_{-}(T)$ then the corresponding triples are projectively equivalent. 
\end{remark}

\bibliography{ref}
\end{document}